\documentclass{article}

     \PassOptionsToPackage{numbers, compress}{natbib}

     \usepackage[preprint]{neurips_2019}

\usepackage[utf8]{inputenc} 
\usepackage[T1]{fontenc}    
\usepackage{hyperref}       
\usepackage{url}            
\usepackage{booktabs}       
\usepackage{amsfonts}       
\usepackage{nicefrac}       
\usepackage{microtype}      
\usepackage{graphicx}
\usepackage{subcaption}
\usepackage[]{algorithm}
\usepackage[]{algorithmic}
\usepackage{booktabs} 
\usepackage{amsmath}
\usepackage{color}
\usepackage{amsfonts}
\RequirePackage{amssymb}

\newcommand{\BlackBox}{\rule{1.5ex}{1.5ex}}  
\newenvironment{proof}{\par\noindent{\bf Proof\ }}{\hfill\BlackBox\\[2mm]}
 
\newtheorem{theorem}{Theorem}[section]

\newtheorem{remark}{Remark}[section]

\newtheorem{definition}{Definition}[section]

\DeclareMathOperator*{\argmax}{arg\,max}
\title{Competitive Algorithms for Online Budget-Constrained Continuous DR-Submodular Problems}

\author{%
  Omid Sadeghi\\
  University of Washington\\
  Seattle, WA 98195\\
  \texttt{omids@uw.edu} \\
  \And
  Reza Eghbali \\
  Tetration Analytics, Cisco Systems, Inc.\\
  Palo Alto, CA 94301\\
  \texttt{reghbali@tetrationanalytics.com} \\
  \AND
  Maryam Fazel \\
  University of Washington\\
  Seattle, WA 98195 \\
  \texttt{mfazel@uw.edu}\\
}

\begin{document}

\maketitle

\begin{abstract}
	In this paper, we study a certain class of online optimization problems, where the goal is to maximize a function that is not necessarily concave and satisfies the Diminishing Returns (DR) property under budget constraints. We analyze a primal-dual algorithm, called the Generalized Sequential algorithm, and we obtain the first bound on the competitive ratio of online monotone DR-submodular function maximization subject to linear packing constraints which matches the known tight bound in the special case of linear objective function. 
\end{abstract}

\section{Introduction}
Online optimization covers a large number of problems including online resource allocation, online bipartite matching \cite{obm-Karp}, the ``Adwords'' problem \cite{adwords-Mehta,adwords-Buchbinder}, online submodular welfare maximization \cite{swm-Lehmann}, online linear programming \cite{onlinecp-Buch} and online concave packing problem \cite{online-Azaretal,online-Eghbali}. One type of algorithms proposed for solving such problems are primal-dual algorithms where the dual variable is updated at each step and is used to get the update rule for the primal variable \cite{primaldual-Buchbinder}.\\
Depending on how much information about the online input is available in advance to the algorithm, online problems have been categorized into adversarial (worst-case) (e.g., in \cite{adwords-Mehta}) and stochastic input models (e.g., in \cite{auction-Kleinberg}) and we consider the former in this paper. In the adversarial model, it is assumed that the algorithm has no knowledge of the online input.
Performance of online algorithms is measured by their competitive ratio defined as the ratio of the value of the objective function at the output of the algorithm to the maximum objective value attained offline. In the worst-case model, one is interested in deriving lower bounds on the competitive ratio of the algorithm that holds for all arbitrary online inputs.\\
In this paper, we discuss a certain class of online optimization problems where the objective function is assumed to satisfy the Diminishing Returns (DR) property under linear packing constraints. We introduce a greedy primal-dual algorithm, called the Generalized Sequential algorithm and we analyze its performance theoretically and numerically under the worst-case input model. Specifically, we make the following contributions:
\begin{itemize}
	\item We introduce the online monotone continuous DR-submodular function maximization subject to linear packing constraints, and specify various online discrete submodular problems whose continuous generalization could be cast in this framework, for example the generalized continuous version of online submodular welfare maximization \cite{swm-Lehmann} and online knapsack constrained monotone submodular function maximization \cite{knapsack-Maehara} are well-known cases.
	\item We introduce the Generalized Sequential algorithm for this class of problems. Denoting the number of linear packing constraints by $n$, we consider the following two cases and we derive competitive ratio bounds for each case:
	\begin{itemize}
		\item $n=1$: In this case, our problem is the generalization of online knapsack constrained monotone submodular function maximization \cite{knapsack-Maehara} to the continous setting. For this online problem, we obtain a competitive ratio of $\frac{1}{1-\alpha+\ln(\frac{U}{L})}$ where $L$ and $U$ are lower and upper bounds on the value-to-weight ratio of the items respectively and $\alpha$ captures the curvature of the DR-submodular utility function.
		\item $n>1$: In this case, our problem generalizes the Adwords problem \cite{adwords-Mehta} and the online linear programming problem \cite{onlinecp-Buch} by allowing the utility function to be DR-submodular rather than linear. For this setting, we obtain the first competitive ratio bound which is optimal in the special cases. Specifically, if the objective function is linear, our problem reduces to the online linear programming and the algorithm achieves the optimal competitive ratio \cite{online-Eghbali,onlinecp-Buch}. If in addition to the linearity of the objective function, the linear packing constraint and the objective function are equal, the problem simplifies to the Adwords problem and we obtain the optimal $1-\frac{1}{e}$ competitive ratio \cite{adwords-Mehta,adwords-Buchbinder} (note that since we allow fractional solutions for the Adwords problem, we do not need the small bids assumption to obtain the optimal competitive ratio).
	\end{itemize}
\end{itemize}
Finally, we present numerical experiments on a class of non-concave DR-submodular utility functions to demonstrate the performance of the Generalized Sequential algorithm.\\
It is noteworthy that although our framework could be interpreted as the generalization of online budgeted discrete submodular problems to the continuous setting, we do not aim to solve the discrete problem itself. In other words, our goal is to solve a class of online budgeted problems where the objective function is originally continuous and DR-submodular. Therefore, we do not round the fractional output of our proposed algorithm. 

\subsection{Notation}
We will use $[m]$ to denote the set $\{1,2,\dots,m\}$. For a matrix $A\in \mathbb{R}^{n\times m}$, we will denote its $i$-th row by $\hat{a}_i^T$ for all $i\in[n]$ and its $t$-th column by $a_t$ for all $t\in[m]$. Also, $a_{i,t}$ corresponds to the $(i,t)$-th entry of the matrix A. We denote the transpose of a matrix $A$ by $A^T$. The inner product of two vectors $x,y\in\mathbb{R}^m$ is denoted by either $\langle x, y \rangle$ or $x^T y$. Also, for two vectors $x,y\in \mathbb{R}^m$, $x\preceq y$ implies that $x_i \leq y_i~\forall i\in[m]$. We use $F^*$ to denote the concave conjugate of a function $F:\mathbb{R}^m \to \mathbb{R}$ which is defined as follows:
\begin{equation*}
	F^*(y)=\inf_x \big(\langle x,y \rangle -F(x)\big)
\end{equation*} 
For a convex set $\mathcal{P}$, the support function of $\mathcal{P}$ is defined in the following:
\begin{equation*}
	\sigma_{\mathcal{P}}(x)=\sup_{y\in \mathcal{P}}\langle x,y\rangle 
\end{equation*}
\section{Diminishing Returns (DR) property}
\begin{definition}\label{def:dr}
	A differentiable function $F:K \rightarrow \mathbb{R}$, $K\subset \mathbb{R}_+^m$, satisfies the Diminishing Returns (DR) property if:
	\begin{equation}
	x\succeq y \Rightarrow \nabla F(x) \preceq \nabla F(y)\nonumber
	\end{equation}
	In other words, $\nabla F$ is an anti-tone mapping from $\mathbb{R}^m$ to $\mathbb{R}^m$.\\
	If $F$ is twice differentiable, DR property is equivalent to the Hessian matrix being element-wise non-positive. Note that for $m=1$, the DR property is equivalent to concavity. However, for $m>1$, concavity implies negative semi-definiteness of the Hessian matrix which is not equivalent to the Hessian matrix being element-wise non-positive.
\end{definition}
A similar property is introduced in \cite{swm-Vondrak} and \cite{dr-Bian} as well and functions satisfying this property are called ``smooth submodular'' and ``DR-submodular'' there respectively. Additionally, \cite{online-Eghbali} defined the DR property for concave functions with respect to a partial ordering induced by a cone and showed that by taking the cone to be $\mathbb{R}_+^m$, Definition \ref{def:dr} is recovered and if the cone of positive semi-definite matrices is considered, the DR property generalizes to matrix ordering as well \cite{onlinematrix-Eghbali}. \cite{dr-Bian} showed that DR-submodular functions are concave along any non-negative direction, and any non-positive direction. In other words, for a DR-submodular function $F$, if $t\geq 0$ and $v\in \mathbb{R}^m$ satisfies $v\succeq 0$ or $v\preceq 0$, we have:
\begin{equation*}
F(x+tv)\leq F(x)+t\langle \nabla F(x),v\rangle 
\end{equation*}
\subsection{Examples of continuous non-concave DR-submodular functions}\label{examples}
\textbf{Multilinear extension of discrete submodular functions. \cite{multilinearmax-calinescu}} A discrete function $f:\{0,1\}^V\rightarrow \mathbb{R}$ is submodular if for all $j\in V$ and $A\subseteq B\subseteq V\setminus\{j\}$, the following holds:
\begin{equation*}
	f(A\cup\{j\})-f(A)\geq f(B\cup\{j\})-f(B)
\end{equation*} 
The multilinear extension $F:[0,1]^V \rightarrow \mathbb{R}$ of $f$ is defined as:
\begin{equation*}
F(x)=\sum_{S\subset V}f(S)\prod_{i\in S}x_i \prod_{j\notin S}(1-x_j)=\mathbb{E}_{S\sim x}[f(S)]
\end{equation*}
Multilinear extensions are extensively used for maximizing their corresponding submodular set function and are known to be a special case of non-concave DR-submodular functions. The Hessian matrix of this class of functions has non-positive off-diagonal entries with zeros on its diagonal. It has been shown that for a large class of submodular set functions, their multilinear extension could be efficiently computed. Weighted matroid rank function, set cover function, probabilistic coverage function, graph cut function and concave over modular function are all examples of such submodular functions (see \cite{sub-Iyer,drapp-Bian} for more examples and details).\\
\textbf{Non-convex/non-concave quadratic functions.} Consider the quadratic function $F(x)=\frac{1}{2}x^T Hx+h^T x+c$. If the matrix $H$ is element-wise non-positive, $F$ would be a DR-submodular function. We use this class of non-concave DR-submodular functions for the numerical experiments.\\
See \cite{dr-Bian,dralg-Bian} for more examples of continuous DR-submodular objective functions.
\section{Problem Statement}
The offline constrained optimization problem is as follows:
\begin{equation}\label{eq:main-prob-con}
\begin{array}{ll}
\mbox{maximize}& \sum_{i=1}^n H_i(\hat{x}_i)\\
\mbox{subject to}& x_t \in F_t\subseteq \mathbb{R}_+^n~\forall t\in[m]\\
&\hat{c}_i^T \hat{x}_i\leq 1~\forall i\in[n]
\end{array}
\end{equation}
where $\hat{x}_i^T\in \mathbb{R}_+^m$ is the $i$-th row and $x_t \in \mathbb{R}_+^n$ is the $t$-th column of the variable matrix $X\in \mathbb{R}_+^{n\times m}$, $\hat{c}_i^T \in \mathbb{R}_+^{m}$ and $c_t \in \mathbb{R}_+^n$ are the $i$-th row and $t$-th column of the cost matrix $C\in \mathbb{R}_+^{n\times m}$ respectively.
For all $i\in[n]$, $H_i: K \to \mathbb{R}$, $K\subset \mathbb{R}_+^m$, is a differentiable monotone non-decreasing DR-submodular function which is zero at the origin (i.e., $H_i(0)=0$). For all $t\in[m]$, $F_t$ is a compact convex constraint set that contains the origin and $\|x\|_2 \leq \lambda$ for all $x\in F_t$.\\
In the online setting, at step $t\in [m]$, $c_t$ and $F_t$ arrive online and the algorithm should choose $x_t \in F_t$ to maximize the overall objective function. Note that at each step $t\in[m]$, the function $H_i~\forall i\in[n]$ is only known over subsets of variables that have already arrived. Thus, we don't have access to the objective function in advance.\\
The penalized formulation of problem \ref{eq:main-prob-con} is the following: 
\begin{equation}\label{eq:main-prob-mod}
\begin{array}{ll}
\mbox{maximize}& \sum_{i=1}^n \big(H_i(\hat{x}_i)+G_i(\hat{c}_i^T \hat{x}_i)\big)\\
\mbox{subject to}& x_t \in F_t\subseteq \mathbb{R}^n~\forall t\in[m]\\
\end{array}
\end{equation}
As an example, if for all $i\in[n]$,
\[  G_i(u) = 
\begin{cases}
0 &\quad\text{if } 0\leq u\leq 1 \\
-\infty &\quad\text{if } u> 1\\
\end{cases}
\] 
i.e., the concave indicator function of the interval $[0,1]$, the above two optimization problems are equivalent.\\
We aim to design differentiable, concave and monotone non-increasing penalty functions $G_i: \mathbb{R}_+ \rightarrow \mathbb{R}~\forall i\in[n]$ and use them in our online algorithm such that the output doesn't violate any of the linear packing constraints.\\ 
Online linear programming \cite{onlinecp-Buch}, the Adwords problem \cite{adwords-Mehta}, single-unit combinatorial auction problem \cite{auction-Huang} and continuous generalization of online knapsack problem \cite{knapsack-Zhou} are all special cases of this framework for linear objective functions.\\
Multiple applications of this framework are provided in Appendix A. 
\subsection{Dual Problem}
The dual problem of the constrained problem \ref{eq:main-prob-con} is as follows: (See Appendix B for the derivation) 
\begin{equation}\label{eq:dual}
\begin{array}{ll}
\mbox{minimize}& \sum_{t=1}^m \sigma_{F_t} \big(\begin{bmatrix} y_{1,t}-z_1c_{1,t} \\ \vdots \\ y_{n,t}-z_n c_{n,t} \end{bmatrix}\big)-\sum_{i=1}^n H_i^*(\hat{y}_i)+\sum_{i=1}^n z_i\\
\mbox{subject to}& z_i\geq 0 ~\forall i\in[n]\\
\end{array}
\end{equation}
where for all $i\in[n]$, $z_i\in \mathbb{R}_+$, $\hat{y}_i^T \in \mathbb{R}^m$ is the $i$-th row of the dual matrix variable $Y$ and $y_{i,t}$ is the $(i,t)$-th entry of this matrix.\\
Karush–Kuhn–Tucker (KKT) conditions can be written as:
\begin{align*}
x_t^* &\in \argmax_{x \in F_t} \ \langle x,\begin{bmatrix} y_{1,t}^*-z_1^*c_{1,t} \\ \vdots \\ y_{n,t}^*-z_n^*c_{n,t} \end{bmatrix} \rangle \\
\hat{y}_i^* &= \nabla H_i(\hat{x}_i^*)~i=1,\dots,n\\
z_i^* &= -G'_i(\hat{c}_i^T \hat{x}_i^*)~i=1,\dots,n\\
\end{align*}
where $G'_i$ is the derivative of the scalar penalty function $G_i$. We remind the reader that we aim to design differentiable penalty functions $G_i~\forall i\in[n]$ and therefore, we have used $G'_i$ in the KKT conditions.\\
We will use these KKT conditions to design the Generalized Sequential Algorithm.
\section{Generalized Sequential algorithm and Competitive Ratio Analysis}
\subsection{Generalized Sequential Algorithm}
Consider the Generalized Sequential algorithm below which outputs $\tilde{x}_t$ at each online step $t\in[m]$.\\

\begin{algorithm}[h]
	\label{alg:seq-mod}
	\caption{Generalized Sequential Algorithm}
	\begin{algorithmic}
		\STATE \textbf{Input}: Penalty functions $G_i~\forall i\in[n]$, $K$ and $m$
		\STATE Initialize $\tilde{X}=0$
		\FOR{$t=1$ {\bfseries to} $m$}
		\STATE 	$c_t, F_t$ arrive online and gradient of $H_i~\forall i\in[n]$ over the first $t$ variables (i.e., all other $m-t$ variables being zero) is accessible
		\STATE $\tilde{x}_t (0)=0$
		\FOR{$k=1$ {\bfseries to} $K$}
		\STATE $v_t (k)=\argmax_{v\in F_t} \langle v, d_t(k-1)\rangle$~~~~~~~~~~~~~\COMMENT{$d_t(k-1)$ defined in equation \ref{d}}
		\STATE $\tilde{x}_t (k)=\tilde{x}_t (k-1) +\frac{1}{K} v_t (k)$
		\ENDFOR
		\STATE \textbf{Output:} $\tilde{x}_t =\tilde{x}_t (K)$
		\ENDFOR
	\end{algorithmic}
\end{algorithm}	
where for all $i\in[n]$, $t\in[m]$ and $k\in\{0,\dots,K\}$:
\begin{align}
\omega_{i,t}(k)&:=\big[[\tilde{x}_1(K)]_i,\dots,[\tilde{x}_{t-1}(K)]_i, [\tilde{x}_t(k)]_i,\underbrace{0,\dots,0}_{m-t~\text{times}}\big]^T\nonumber\\
[d_t(k)]_i&:=\nabla_t H_i \big(\omega_{i,t} (k)\big)+c_{i,t}G'_i \big(\hat{c}_i^T \omega_{i,t}(k)\big)\label{d}
\end{align} 
In the above definitions, we have used the notation $[u]_i$ to denote the $i$-th entry of the vector $u$ and $\nabla_t$ denotes the $t$-th entry of the gradient vector.\\
At each online step $t\in[m]$, the algorithm performs a total of $K$ Frank-Wolfe updates in its inner loop where in each of these updates, a linear maximization problem over the set $F_t$ is solved. Note that in our applications, $F_t$ is usually a box constraint or the simplex and therefore, the corresponding linear maximization problem could be solved efficiently. See \cite{lacoste2016convergence} for more details about using Frank-Wolfe for non-convex objectives.\\
The Generalized Sequential algorithm reduces to the Sequential algorithm in \cite{online-Eghbali} for $K=1$ and hence the name. Additionally, this algorithm could be interpreted as the online counterpart of the offline Frank-Wolfe variant proposed in \cite{dr-Bian} for solving offline constrained continuous DR-submodular optimization problems.
Note that at step $t\in[m],k\in[K]$, if we set $\hat{y}_i^*=\nabla H_i \big(\omega_{i,t} (k-1)\big)$ and $z_i^*=-G'_i(\hat{c}_i^T \omega_{i,t}(k-1))$, the update rule for $v_t(k)$ is similar to the KKT condition for $x_t^*$. In other words, the Generalized Sequential algorithm uses $\nabla H_i(\omega_{i,t}(k-1))$ and $-G'_i(\hat{c}_i^T \omega_{i,t}(k-1))$ as the current estimate of $\hat{y}_i^*$ and $z_i^*$ respectively and using them, the algorithm obtains $v_t(k)$ to improve the estimate of $x_t^*$.\\
We define:
\begin{align*}
{\rm ALG}&:= \sum_{i=1}^n H_i(\omega_{i,m}(K))\\
{\rm P_{gseq}} &:= \sum_{i=1}^n \big(H_i(\omega_{i,m}(K))+G_i(\hat{c}_i^T \omega_{i,m}(K))\big)
\end{align*}
${\rm ALG}$ and ${\rm P_{gseq}}$ are the objective value of problems \ref{eq:main-prob-con} and \ref{eq:main-prob-mod} at the end of the algorithm respectively. Note that since $0\in F_t~\forall t\in[m]$, whenever $d_t (k-1)\preceq 0$, the algorithm would assign zero to $v_t (k)$ and therefore, ${\rm P_{gseq}}\geq 0$.\\
\subsection{Competitive Ratio Analysis}
First, we remind the reader that $H_i~\forall i\in[n]$ are DR-submodular and not necessarily concave. On the other hand, $G_i~\forall i\in[n]$ are concave penalty functions. In order to derive the competitive ratio, we make the following smoothness assumption about the functions:\\
\textbf{Assumption 1:} For all $i\in[n]$, functions $H_i$ and $G_i$ have an $L$-Lipschitz gradient, i.e., for all $x\in K$ and $u\in \mathbb{R}_+^m$ where $u\succeq 0$ or $u\preceq 0$, the following holds:
\begin{equation}
H_i(x+u)-H_i(x)\geq \langle u, \nabla H_i (x)\rangle -\frac{L}{2}\|u\|^2\nonumber
\end{equation} 
Also, for all $x\in \mathbb{R}$ and $v\in \mathbb{R}$, we have:
\begin{equation}
G_i(x+v)-G_i(x)\geq v G'_i (x) -\frac{L}{2}v^2\nonumber
\end{equation}
We also define the parameter $\alpha$ as follows:
\begin{definition}\label{def:alpha}
	For all $i\in [n]$, $\alpha_{H_i}$ is defined as:
	\begin{align*}
	\alpha_{H_i} &:= \sup_{u: \hat{c}_i^T u\leq 1} \{\beta~|~H_i^*\big(\nabla H_i(u)\big)\geq \beta H_i(u)\}\\
	&= \sup_{u: \hat{c}_i^T u\leq 1} \{\beta~|~\langle \nabla H_i(u),u\rangle\geq (1+\beta)H_i(u)\}\\
	&= \inf_{u: \hat{c}_i^T u\leq 1}\frac{\langle \nabla H_i(u),u\rangle}{H_i(u)} -1
	\end{align*}
	Since $H_i$ is monotone non-decreasing, $0\leq \langle \nabla H_i(u),u\rangle$ holds. Additionally, because $H_i$ satisfies the DR property and $H_i(0)=0$, we have $\langle \nabla H_i(u),u\rangle \leq H_i(u)$. Thus, $-1 \leq \alpha_{H_i}\leq 0$ always holds.\\
\end{definition}
The definition above is inspired by the definition of $\alpha$ in \cite{online-Eghbali}. The parameter $\alpha$ characterizes the curvature of the function. In fact, $\alpha$ of the multilinear extension of a submodular function and the total curvature of the underlying submodular set function are related as follows:
\begin{remark}\label{rm1}
	\textbf{Connection between total curvature of a submodular function and $\alpha$:}\\
	Recall that for a non-negative normalized monotone non-decreasing submodular function $f: 2^V \rightarrow \mathbb{R}~\forall i\in [n]$, total curvature is defined as \cite{curvature-Conforti}:
	\begin{equation}\label{eq:curv}
	\kappa_{f}=1- \min_{j:f(j)\neq 0} \frac{f(j|(V\setminus j))}{f(j)}=1-\min_{S\subset V\setminus\{j\}, f(j)\neq 0}\frac{f(j|S)}{f(j)}\nonumber
	\end{equation}
	If we denote the multilinear extension of this function by $F$, the following holds:
	\begin{equation*}
		\alpha_{F} \geq -\kappa_{f}
	\end{equation*}
	See Appendix C for the proof.
\end{remark}

If $n=1$, we design the penalty function $G_1$ as follows:
\[  G_1(u) = 
\begin{cases}
-L_1 u &\quad\text{if } 0\leq u< \frac{1}{\ln\big(\frac{U_1 e}{L_1}\big)} \\
-\frac{1}{\ln\big(\frac{U_1 e}{L_1}\big)}\frac{L_1}{e}(\frac{U_1 e}{L_1})^u &\quad\text{if } u\geq \frac{1}{\ln\big(\frac{U_1 e}{L_1}\big)} \\
\end{cases}
\] 
If $n>1$, for all $i\in[n]$, we define the penalty function $G_i$ in the following:
\begin{align*}
G_i(u)&=\frac{L_i}{(e-1)\ln(1+\frac{U_i(e-1)}{L_i})}\big(1-(1+\frac{U_i(e-1)}{L_i})^u\big)+\frac{L_i}{e-1}u
\end{align*}
where for all $i\in[n]$, $U_i$ and $L_i$ are defined as follows:
\begin{align*}
U_i &= \max_{t\in[m]}\frac{\sup_{x\in \mathbb{R}^m:~\hat{c}_i^T x=1}\nabla_t H_i (x)}{c_{i,t}}\\
L_i &= \min_{t\in[m]}\frac{\inf_{x\in \mathbb{R}^m:~\hat{c}_i^T x\leq 1}\nabla_t H_i (x)}{c_{i,t}}
\end{align*}
Roughly speaking, $U_i$ and $L_i$ are upper and lower bounds for the value-to-weight ratio of the items arriving online respectively. We are assuming that these upper and lower bounds are available offline to design the penalty functions.
Our design for the penalty function for $n=1$ is inspired by the threshold function proposed by \cite{knapsack-Zhou}. In the $n>1$ case, our penalty functions are inspired by the allocation rule of the primal-dual algorithm for the Adwords problem \cite{adwords-Buchbinder}. In both cases, the penalty functions are designed such that for $u\geq 1$, the algorithm assigns zero and thus, $G'_i (1)= -U_i$ holds. Thus, the algorithm's assignments would not violate the budget constraints.\\
If we denote the optimal values of the original constrained problem \ref{eq:main-prob-con} and its dual problem \ref{eq:dual} by ${\rm OPT}$ and ${\rm D}^*$ respectively, ${\rm OPT}\leq {\rm D}^*$ holds due to weak duality.\\
Now, we have all the required tools to obtain the competitive ratio bounds.
\begin{theorem}\label{thm1}
	For $n>1$, if Assumption 1 holds and $K\rightarrow \infty$, then for the Generalized Sequential algorithm, we have:
	\begin{equation}
	\frac{\rm ALG}{\rm OPT} \geq \frac{\rm ALG}{\rm D^*}\geq \big({\max_{i\in[n]}\big\{-\alpha_{H_i}+\ln(1+\frac{U_i(e-1)}{L_i})\frac{e}{e-1}\big\}}\big)^{-1}
	\end{equation}
	This bound is tight in several known special cases. For the Adwords problem, since $U_i =L_i =1$ and $\alpha_{H_i}=0$ for all $i\in[n]$, competitive ratio of $1-\frac{1}{e}$ is obtained which is optimal \cite{adwords-Mehta}. 
	Additionally, for online linear programming, considering that $\alpha_{H_i}=0~\forall i\in[n]$, we obtain $\big({\max_{i\in[n]}\ln(1+\frac{U_i(e-1)}{L_i})}\big)^{-1}\times(1-\frac{1}{e})$ as the competitive ratio bound which is known to be optimal \cite{online-Eghbali,onlinecp-Buch}.
\end{theorem}
\begin{proof}
	See Appendix D for the proof.
\end{proof}
\begin{remark}
	For $n>1$, if we allow all the linear packing constraints to be violated by at most $\epsilon$, by modifying the penalty function for all $i\in[n]$ to
	\begin{equation*}
	G_i(u)=\frac{L_i (1+\epsilon)}{(e-1)\ln(1+\frac{U_i(e-1)}{L_i})}\big(1-(1+\frac{U_i(e-1)}{L_i})^{\frac{u}{1+\epsilon}}\big)+\frac{L_i}{e-1}u
	\end{equation*}
	competitive ratio improves to $(1+\epsilon)\times\big({\max_{i\in[n]}\big\{-(1+\epsilon)\alpha_{H_i}+\ln(1+\frac{U_i(e-1)}{L_i})\frac{e}{e-1}\big\}}\big)^{-1}$
\end{remark}
\begin{theorem}\label{thm2}
	For $n=1$, if Assumption 1 holds and $K\rightarrow \infty$, then for the Generalized Sequential algorithm, we have:
	\begin{equation}
	\frac{\rm ALG}{\rm OPT}\geq \frac{\rm ALG}{\rm D^*}\geq \frac{1}{1-\alpha_{H_1}+\ln(\frac{U_1}{L_1})}
	\end{equation}
	For the online linear knapsack problem, since $\alpha_{H_1}=0$, competitive ratio of $\frac{1}{1+\ln(\frac{U_1}{L_1})}$ is obtained which is optimal \cite{knapsack-Zhou} (note that because we allow fractional assignments, we do not need the small bids assumption to obtain the optimal competitive ratio).
\end{theorem}
\begin{proof}
	See Appendix E for the proof.
\end{proof}
\begin{remark}
	For $n=1$, if we allow the linear packing constraint to be violated by at most $\epsilon$, modifying the penalty function to
	\[  G_1(u) = 
	\begin{cases}
	-L_1 u &\quad\text{if } 0\leq u< \frac{1+\epsilon}{\ln\big(\frac{U_1 e}{L_1}\big)} \\
	-\frac{1+\epsilon}{\ln\big(\frac{U_1 e}{L_1}\big)}\frac{L_1}{e}(\frac{U_1 e}{L_1})^{\frac{u}{1+\epsilon}} &\quad\text{if } u\geq \frac{1+\epsilon}{\ln\big(\frac{U_1 e}{L_1}\big)} \\
	\end{cases}
	\] 
	we obtain the improved competitive ratio of $\frac{1+\epsilon}{-(1+\epsilon)\alpha_{H_1}+\ln(\frac{U_1e}{L_1})}$
\end{remark}
Theorems \ref{thm1} and \ref{thm2} provide the first competitive ratio bounds that generalize the results of \cite{online-Azaretal,online-Eghbali} for the concave case to general continuous DR-submodular objective functions which are not necessarily concave.
\section{Experiments}
We defined $F_t=\{x\in \mathbb{R}^n:0\preceq x\preceq \bf{1}\}$ for all $t\in[m]$ and we randomly generated monotone non-convex/non-concave quadratic functions of the form $F(x)=\frac{1}{2}x^T Hx+h^T x$ (see \ref{examples}) where $H\in \mathbb{R}^{m\times m}$ is a random matrix with uniformly distributed non-positive entries in $[-100,0]$ and $h=-H^T \bf{1}$ to make the gradient non-negative. Therefore, the utility functions are of the form $F(x)=(\frac{1}{2}x-{\bf{1}})^T Hx$. We set the linear packing constraints to be of the form $Cx\preceq \bf{1}$ where $C\in \mathbb{R}^{n\times m}$ has uniformly distributed entries in $[0,1]$. We set $m=100$ and $K=50$. For all $i\in[n]$, the lower and upper bounds $L_i$ and $U_i$ were optimized by the input data. We ran the Generalized Sequential algorithm for both cases of $n=1$ and $n>1$ (note that the penalty function defined in these two cases were different) and in order to compute the competitive ratio, we divided the output of the algorithm to the offline optimal solution computed by the Frank-Wolfe variant algorithm of \cite{dr-Bian} with $K=50$. The average performance of the Generalized sequential algorithm over $10$ repeated experiments is summarized in Table \ref{table:1}. All codes were implemented in Python $3.7$ and the program was executed on a standard laptop computer ($2.30$GHz CPU, $16.0$GB Memory).
\begin{table}[h]
	\begin{subtable}{.5\linewidth}
	\centering
	\begin{tabular}{|c|c|} 
	\hline
	\text{Quantity} & \text{Value} ($\%$)\\ [0.5ex] 
	\hline
	\text{Competitive Ratio} & 64.33\\ 
	\hline
	\text{Budget Usage} & 74.95 \\
	\hline
	\end{tabular}
\caption{ $n=1, m=100, K=50$}
	\end{subtable}
	\begin{subtable}{.5\linewidth}
		\centering
		\begin{tabular}{|c|c|} 
			\hline
			\text{Quantity} & \text{Value} ($\%$)\\ [0.5ex] 
			\hline
			\text{Competitive Ratio} & 58.27\\ 
			\hline
			\text{Budget $1$ Usage} & 65.68 \\
			\hline
			\text{Budget $2$ Usage} & 58.06 \\
			\hline
			\text{Budget $3$ Usage} & 66.83 \\
			\hline
			\text{Budget $4$ Usage} & 65.11 \\
			\hline
			\text{Budget $5$ Usage} & 74.75 \\
			\hline
		\end{tabular}
		\caption{ $n=5, m=100, K=50$}
	\end{subtable} 
	\caption{Performance of the Generalized Sequential Algorithm}
	\label{table:1}
\end{table}
Table \ref{table:1} shows that the output of the Generalized Sequential algorithm is not using all of the available budget which is natural in the adversarial input model. In other words, considering that no information about the online input is available, in order to attain a guaranteed competitive ratio, the algorithm needs to be overly cautious so that it does not miss valuable items that are arriving in the later steps due to exhausting all of the budget in the earlier stages.
\section{Related Work}
\textbf{Offline submodular maximization.} Consider the problem $\max_{x\in \mathcal{P}}F(x)$ where $F:\mathbb{R}^m \rightarrow \mathbb{R}$ is a non-negative monotone DR-submodular function and $\mathcal{P}$ is a down-closed convex set in the positive orthant. In \cite{drmax-Calinescu}, the special case of continuous relaxation of the discrete submodular function maximization problem subject to a matroid constraint is considered, the following variant of the Frank-Wolfe algorithm called the Continuous Greedy is used and a $(1-\frac{1}{e})$ approximation ratio is obtained. As it was mentioned in section \ref{examples}, the multilinear extension of a submodular function satisfies the DR property.
\[\begin{array}{ll}
dy/dt&=v_{\rm max}(y)\\
v_{\rm max}(y)&=\argmax_{v\in \mathcal{P}}\langle v,\nabla F(y)\rangle
\end{array}
\]
In this algorithm, $y(1)=\int_0^1 v_{\rm max}(y(\tau))d\tau$ is the output. Our Generalized Sequential algorithm is in fact the online counterpart of the discretized Continuous Greedy algorithm (i.e., K is the number of discrete steps in our algorithm)\\
\cite{dr-Bian} obtained a similar approximation ratio for general continuous DR-submodular functions. \cite{knapsack-Kulik} also exploited the same algorithm to obtain a $(1-\frac{1}{e})$ approximation ratio for submodular maximization subject to multiple linear constraints. Later on, the Continuous Greedy algorithm has been generalized to obtain approximation ratios for both monotone and non-monotone continuous submodular functions \cite{dr-Bian,gsubmax-chekuri,mwu-Chekuri,gsubmax-Feldman,gsubmax-Ene,gsubmax-Buch}. See \cite{survey-Krause,survey-Buchbinder} for a thorough overview of offline submodular maximization problems and algorithms.\\
\textbf{Online knapsack problem.} Consider the problem $\max_{x:c^Tx\leq b} f(x)$ where $x=[x_1,\dots,x_m]^T\in \{0,1\}^m$. In the online setting, at step $t\in[m]$, $t$-th item arrives and $c_t$ along with the value of the function $f$ over subsets of $\{1,\dots,t\}$ is revealed. The algorithm should decide whether to choose this item. \cite{knapsack-Marchetti} showed that in the adversarial setting, there exists no online algorithm achieving any non-trivial competitive ratio for this problem. \cite{knapsack-Zhou} considered the case where $f(x)=d^T x;~d\in\mathbb{R}^m$ and proved that under the additional assumptions that for all $t\in[m]$, $c_t \ll b$ and $L\leq \frac{d_t}{c_t}\leq U$, there exists an algorithm that achieves the competitive ratio of $\frac{1}{1+ln(\frac{U}{L})}$ and is provably optimal. \cite{knapsack-Maehara} generalized this algorithm for the case that the function $f$ is submodular and obtained a $\frac{1}{(1+\kappa_f+O(\epsilon))(1+ln(\frac{U}{L}))}$ competitive ratio where $L\leq \frac{f(t|S)}{c_t}\leq U~\forall t\in[m],S\subset \{1,\dots,m\}\setminus \{t\}$ and $\kappa_f$ is the total curvature of $f$ \cite{curvature-Conforti}. Note that if we apply our Generalized Sequential algorithm for $n=1$ to the multilinear extension of the function $f$ (which we denote by $F$) and allow fractional assignments of items, we obtain the competitive ratio $\frac{1}{1-\alpha_{F}+\ln(\frac{U}{L})}$ and because $\alpha_F \geq -\kappa_f$ holds by Remark \ref{rm1}, our bound improves upon the result of \cite{knapsack-Maehara}.\\
\textbf{Submodular secretary problems.} In this class of problems introduced by \cite{secretary-Bateni,secretary-Gupta}, $m$ items are presented to the algorithm in random order. Upon arrival of an item, the algorithm should irrevocably decide whether to accept the current item. The goal is to maximize a monotone submodular function $f:\{0,1\}^m \rightarrow \mathbb{R}$ subject to cardinality, matching or linear packing constraints. See \cite{secretary-Kesselheim,survey-Krause} for a comprehensive overview of submodular secretary problems. Note that in the submodular secretary problem, the input is assumed to be stochastic while in our framework, the adversarial input model has been considered.
\section{Conclusion}
In this paper, we considered a class of online optimization problems, where the objective function is monotone DR-submodular under linear packing constraints. We specified various online discrete submodular problems whose continuous generalization could be cast in our framework (see Appendix A). We proposed the Generalized Sequential algorithm for solving such problems and we obtained competitive ratio bounds for this algorithm. Finally, we demonstrated the effectiveness of our algorithm through numerical experiments on a certain class of continuous DR-submodular functions.
\appendix
\addcontentsline{toc}{section}{Appendices}
\section*{Appendices}
\section{Motivating Applications}
There are a number of online budgeted discrete submodular problems whose continuous generalization could be cast in our framework. We have listed a number of these applications below:\\
\textbf{Online Knapsack Constrained Continuous DR-submodular Maximization.} In the discrete problem considered in \cite{knapsack-Maehara}, there is a ground set of elements $V$ and a budget constraint $b \in \mathbb{R}_+$. At step $t\in [m]$, an element $v\in V$ with the corresponding cost $c(v)\in \mathbb{R}_+$ arrives online and we should decide whether to choose $v$. The overall objective is as follows:
\[
\begin{array}{ll}
\mbox{maximize}& f(V')\\
\mbox{subject to}& \sum_{v \in V'} c(v)\leq b\\
\end{array}
\]
where $f:2^V \rightarrow \mathbb{R}_+$ is a monotone non-decreasing submodular function and $V'$ is the set of chosen elements. Note that at each step, value of the function is only known over subsets of items which have already arrived.\\
Consider the continuous relaxation of this problem where at each step, we are allowed to take a fraction of the arriving element. This problem could be formulated as:
\[
\begin{array}{ll}
\mbox{maximize}& F(x)\\
\mbox{subject to}& \sum_{t=1}^m c_t x_t\leq b\\
&0\leq x_t \leq 1 ~\forall t\in [m]\\
\end{array}
\]
where $x=[x_1,\dots,x_m]^T$, $c_t \in \mathbb{R}_+$ is the cost corresponding to the $t$-th arriving element and $F:[0,1]^m \rightarrow \mathbb{R}_+$ is the multilinear extension of the function $f$.\\
\textbf{Online Generalized Maximum Coverage Problem.} In this problem, there are $m$ subsets $C_1,\dots,C_m$ of the ground set $V$ with corresponding costs $c_1,\dots,c_m$ that are arriving one by one. At step $t\in[m]$, subset $C_t$ could be chosen with confidence level $x_t \in[0,1]$ and the set of covered elements when choosing $C_t$ with confidence $x_t$ is modeled with a monotone normalized covering function $p_t: [0,1]\to 2^{C_t}$ which is not known in advance and is revealed online. The goal is to choose subsets from $C_1,\dots,C_m$ with confidence level to maximize the overall number of covered elements $|\bigcup_{t=1}^m p_t (x_t)|$ while satisfying the budget constraint $\sum_{t=1}^m c_tx_t\leq b$. The problem could be formulated as follows:
\[
\begin{array}{ll}
\mbox{maximize}& |\bigcup_{t=1}^m p_t (x_t)|\\
\mbox{subject to}& \sum_{t=1}^m c_t x_t\leq b\\
&0\leq x_t \leq 1 ~\forall t\in [m]\\
\end{array}
\]
\textbf{Online Continuous DR-submodular Welfare Maximization.} In the submodular welfare problem, there is a set $M=\{1,\dots,m\}$ of $m$ items and a set $N=\{1,\dots,n\}$ of $n$ agents. Each agent $i\in N$ has a valuation function $f_i :2^M \rightarrow \mathbb{R}_+$ over subsets of items. Valuation functions are assumed to be submodular and monotone non-decreasing. In this problem, the goal is to partition the items among the agents as $S=(S_1,\dots,S_n)$, where $S_s \cap S_t =\emptyset~ \forall s,t \in N$ and $\cup_{s=1}^n S_s =M$, in a way that the value of the partition $f(S)=\sum_{i=1}^n f_i (S_i)$ is maximized \cite{swm-Vondrak}. Now, consider the continuous relaxation of this problem in the online setting: Each agent has a valuation function $F_i :[0,1]^M \rightarrow \mathbb{R}_+$ which is the multilinear extension of the submodular function $f_i$. At step $t\in M$, item $t$ arrives and the valuations of agents over subsets of items $\{1,\dots,t\}$ are accessible. The algorithm should assign item $t$ fractionally among the agents to maximize the aggregate valuation. The problem could be written as:
\begin{equation*}
	\begin{array}{ll}
		\mbox{maximize}& \sum_{i=1}^n f_i (\hat{x}_i)\\
		\mbox{subject to}& \sum_{i=1}^n x_{it}\leq 1 ~\forall t\in M\\
		&x_{it}\geq 0~\forall i\in[n],t\in[m]
	\end{array}
\end{equation*}
where $\hat{x}_i =[x_{i1},\dots,x_{im}]^T$. Note that in this problem, there are no budget constraints.\\
\textbf{Online DR-Submodular Generalized Assignment Problem.} In this problem, there are $n$ bins and $m$ items. Each bin $i\in[n]$ has an associated collection of feasible sets given by the knapsack constraint $\mathcal{F}_i=\{S\subset [m]~:~\sum_{j\in S}c_{ij}\leq b_i\}$ and a monotone submodular valuation function $f_i:\{0,1\}^m \rightarrow \mathbb{R}_+$ which captures the diversity of the items in each bin. In the online setting, the set of items $t\in[m]$ arrive one by one and upon arrival of each item $t$, $c_{it}$ and values of the functions $f_i$ over subsets of $\{1,\dots,t\}$ for all $i\in[n]$ are revealed. The goal is to partition the items among the bins so as the aggregate valuation of the partition is maximized. If the valuation function $f_i$ is modular for all $i\in[n]$, this problem reduces to the Generalized Assignment Problem (GAP) \cite{swm-Vondrak}. Now, consider the continuous relaxation of this online problem where fractional assignments of items to bins are possible. The problem could be formulated as: 
\begin{equation*}
	\begin{array}{ll}
		\mbox{maximize}& \sum_{i=1}^n F_i (\hat{x}_i)\\
		\mbox{subject to}& \sum_{i=1}^n x_{it}\leq1 ~\forall t\in [m]\\
		&\sum_{t=1}^m c_{it}x_{it}\leq b_i ~\forall i\in [n]\\
	\end{array}
\end{equation*}
where $\hat{x}_i =[x_{i1},\dots,x_{im}]^T$ and $F_i: [0,1]^m \rightarrow \mathbb{R}_+$ is the multilinear extension of the submodular valuation function $f_i$ of the $i$-th bin.\\
\section{Derivation of the Dual Problem}
Let \[  I_{F_t}(x) = 
\begin{cases}
0 &\quad\text{if } x\in F_t \\
\infty &\quad\text{o.w.}\\
\end{cases}
\] 
i.e., the convex indicator function of the set $F_t$.\\
Remember the offline constrained optimization problem:
\begin{equation}\label{eq:main-prob-con2}
	\begin{array}{ll}
		\mbox{maximize}& \sum_{i=1}^n H_i(\hat{x}_i)\\
		\mbox{subject to}& x_t \in F_t\subseteq \mathbb{R}_+^n~\forall t\in[m]\\
		&\hat{c}_i^T \hat{x}_i\leq 1~\forall i\in[n]
	\end{array}
\end{equation}
We derive the dual of problem \ref{eq:main-prob-con2} as follows:
\begin{align*}
	g(\hat{y}_i,z_i)&= \inf_{\hat{d}_i,\hat{e}_i,X} \sum_{i=1}^n -H_i (\hat{d}_i)+\sum_{i=1}^n\hat{y}_i^T (\hat{d}_i-\begin{bmatrix}x_{i,1} \\ \vdots \\ x_{i,m} \end{bmatrix})\\
	&+\sum_{i=1}^n z_i\big(\hat{c}_i^T \begin{bmatrix}x_{i,1} \\ \vdots \\ x_{i,m} \end{bmatrix}-1\big)+\sum_{t=1}^m I_{F_t} (x_t)\\
	&= \sum_{i=1}^n \inf_{\hat{d}_i} \big(\hat{y}_i^T \hat{d}_i-H_i (\hat{d}_i)\big)-\sum_{i=1}^n z_i\\
	&+\sum_{t=1}^m \inf_{x_t \in F_t} \big(I_{F_t}(x_t)-\langle \underbrace{\begin{bmatrix} y_{1,t}-z_1c_{1,t} \\ \vdots \\ y_{n,t}-z_n c_{n,t} \end{bmatrix}}_{v_t},x_t\rangle\big)\\
	&= \sum_{i=1}^n H_i^*(\hat{y}_i)-\sum_{i=1}^n z_i-\sum_{t=1}^m \sup_{x_t \in F_t} \big(\langle v_t,x_t\rangle-I_{F_t}(x_t)\big)\\
	&= \sum_{i=1}^n H_i^*(\hat{y}_i)-\sum_{i=1}^n z_i-\sum_{t=1}^m \sigma_{F_t} (v_t)
\end{align*}
where $\hat{y}_i=[y_{i,1},\dots,y_{i,m}]^T$, $\sigma_{F_t}(u)=\sup_{w \in F_t} w^T u$ is the support function of the set $F_t$ and $H_i^* (u)=\inf_{w} \big(w^Tu -H_i (w)\big)$ is the concave conjugate function of $H_i$.
Therefore, the dual problem is:
\[
\begin{array}{ll}
\mbox{minimize}& \sum_{t=1}^m \sigma_{F_t} \big(v_t\big)-\sum_{i=1}^n H_i^*(\hat{y}_i)+\sum_{i=1}^n z_i\\
\mbox{subject to}& z_i\geq 0 ~\forall i\in[n]\\
\end{array}
\]
\section{Connection between total curvature of a submodular function and $\alpha$}
First, note that $F$, i.e., the multilinear extension of the discrete submodular function $f$, satisfies the DR property and $F(0)=0$. Since $F$ is linear in each of its arguments, we can write:
\begin{equation}\label{eq:alph-c-1}
	\langle \nabla F(x),x\rangle =\sum_t \mathbb{E}_{S\sim x}[f(S\cup\{t\})-f(S\setminus\{t\})]x_t
\end{equation}
Depending on whether $t\in S$ or not, one of the terms $\big(f(S\cup\{t\})-f(S)\big)$ or $\big(f(S)-f(S\setminus\{t\})\big)$ would be zero. So, by definition of total curvature of $f$, i.e., $\kappa_f$, we have:
\begin{align}
	f(S\cup\{t\})-f(S\setminus\{t\})&=\big(f(S\cup\{t\})-f(S)\big)\nonumber\\
	&+\big(f(S)-f(S\setminus\{t\})\big)\nonumber \\
	&\geq (1-\kappa_{f})f(\{t\})\label{eq:alph-c-2}
\end{align}
Combining \ref{eq:alph-c-1} and \ref{eq:alph-c-2}, we have:
\begin{equation}\label{eq:alph-c-3}
	\langle \nabla F(x),x\rangle \geq (1-\kappa_{f})\sum_t x_t f(\{t\})
\end{equation}
Defining $\hat{x}_t=[x_1,\dots,x_t,0,\dots,0]^T$, we can write:
\begin{align*}
	F(x)&=\sum_t \big( F(\hat{x}_t)-F(\hat{x}_{t-1})\big)\\
	&= \sum_t x_t \nabla_t F(\hat{x}_{t-1})
\end{align*}
Since $f(\{t\})=F(1_t)=F(1_t)-F(0)=\nabla_t F(0)$, using the DR property of the function $F$, $\nabla_t F(0) \geq \nabla_t F(\hat{x}_{t-1})$ and therefore, we have:
\begin{equation}\label{eq:alph-c-4}
	F(x) \leq \sum_t x_t f(\{t\})
\end{equation}
Combining \ref{eq:alph-c-3} and \ref{eq:alph-c-4}, we conclude:
\begin{align*}
	\langle \nabla F(x),x\rangle &\geq (1-\kappa_{f})F(x)\\
	\frac{\langle \nabla F(x),x\rangle}{F(x)}&\geq (1-\kappa_{f})\\
	\alpha_{F} &\geq -\kappa_{f}
\end{align*}
As a corollary, since $\alpha_{F} \in [-1,0]$ and $\kappa_{f} \in [0,1]$, if $\kappa_{f} =0$ (i.e., $f$ is modular), we can conclude that $\alpha_{f} =0$ as well.
\section{Proof of Theorem 4.1}
For all $i\in[n]$, using the mean-value theorem, we have:
\begin{align*}
	H_i(\omega_{i,m}(K))&= \sum_{t=1}^m \big(H_i(\omega_{i,t}(K))-H_i(\omega_{i,t}(0))\big)\\
	&= \sum_{t=1}^m \tilde{x}_{i,t}\nabla_t H_i (u_t)
\end{align*}
where $u_t \in \mathbb{R}^m$ and $\omega_{i,t}(0)\preceq u_t \preceq \omega_{i,t}(K)$.
Thus, we can write:
\begin{equation}\label{ell}
	L_i \leq \frac{H_i(\omega_{i,m}(K))}{\hat{c}_i^T \omega_{i,m}(K)}\\
\end{equation}
Considering that $\|x\|_2 \leq \lambda$ holds for all $x\in F_t$ and $t\in[m]$, we can write:
\allowdisplaybreaks
\begin{align*}
	P_{\rm gseq} &= \sum_{i=1}^n \big(H_i(\omega_{i,m}(K))+G_i((\hat{c}_i^T \omega_{i,m}(K))\big)\\
	&= \sum_{i=1}^n \sum_{t=1}^m \big((H_i (\omega_{i,t}(K))-H_i (\omega_{i,t}(0))\big)+ \sum_{i=1}^n \sum_{t=1}^m \big(G_i (\hat{c}_i^T\omega_{i,t}(K))-G_i \hat{c}_i^T\omega_{i,t}(0))) \big)\\
	&= \sum_{i=1}^n \sum_{t=1}^m \sum_{k=1}^K \big((H_i (\omega_{i,t}(k))-H_i (\omega_{i,t} (k-1))\big)+ \sum_{i=1}^n \sum_{t=1}^m \sum_{k=1}^K\big(G_i ((\hat{c}_i^T\omega_{i,t}(k))-G_i (\hat{c}_i^T\omega_{i,t}(k-1))) \big)\\
	&\overset{\text{(a)}}\geq \sum_{i=1}^n \sum_{t=1}^m \sum_{k=1}^K \big( \frac{1}{K}[v_t (k)]_i \nabla_t H_i (\omega_{i,t} (k-1))-\frac{L}{2K^2}[v_t (k)]_i^2\big)\\
	&+ \sum_{i=1}^n \sum_{t=1}^m \sum_{k=1}^K \big( \frac{1}{K}c_{i,t}[v_t (k)]_i G'_i (\hat{c}_i^T\omega_{i,t} (k-1))-\frac{Lc_{i,t}^2}{2K^2}[v_t (k)]_i^2 \big)\\
	&\overset{\text{(b)}}\geq \sum_{i=1}^n \sum_{t=1}^m \sum_{k=1}^K \big( \frac{1}{K}[v_t (k)]_i \nabla_t H_i (\omega_{i,t} (k-1))\big)-\frac{Lm\lambda^2}{2K}\\
	&+ \sum_{i=1}^n \sum_{t=1}^m \sum_{k=1}^K \big( \frac{1}{K}c_{i,t}[v_t (k)]_i G'_i (\hat{c}_i^T\omega_{i,t} (k-1))\big)-\frac{Lm\lambda^2}{2K}\\
	&= \sum_{t=1}^m \sum_{k=1}^K \frac{1}{K} \langle v_t (k),d_t (k-1)\rangle-\frac{Lm\lambda^2}{K}\\
	&\overset{\text{(c)}}= \sum_{t=1}^m \sum_{k=1}^K \frac{1}{K} \sigma_{F_t}(d_t (k-1)) -\frac{Lm\lambda^2}{K}\\
	&\overset{\text{(d)}}\geq \sum_{t=1}^m \sigma_{F_t}(\frac{1}{K}\sum_{k=1}^K d_t (k-1)) -\frac{Lm\lambda^2}{K}
\end{align*}
where (a) is due to Assumption 1, (b) follows from $\|x\|_2 \leq \lambda~\forall x\in F_t, t\in[m]$, (c) uses the update rule of the Generalized Sequential algorithm and (d) is a result of subadditivity of the support function $\sigma_{F_t}$.\\
Using the DR assumption for $H_i$ and $G_i$, for all $i\in[n]$, $t\in[m]$ and $x\in F_t$, we can write: (we remind the reader that for scalar functions such as our concave penalty functions $G_i$, the DR property is equivalent to concavity)
\allowdisplaybreaks
\begin{align}
	d_t (K)&\preceq d_t (k-1)\nonumber\\
	d_t (K)&\preceq \frac{1}{K}\sum_{k=1}^K d_t (k-1)\nonumber\\
	x^Td_t (K)&\leq \frac{1}{K}x^T\sum_{k=1}^K d_t (k-1)\label{eq:thm-proof-k-2}
\end{align}
Taking supremum of $\ref{eq:thm-proof-k-2}$ over all $x\in F_t$, we obtain:
\begin{equation*}
	\sigma_{F_t} (d_t (K))\leq \sigma_{F_t}(\frac{1}{K}\sum_{k=1}^K d_t (k-1))\label{eq:thm-proof-k-3}
\end{equation*}
Therefore, we have:
\begin{align}
	{\rm ALG}&\geq \sum_{t=1}^m \sigma_{F_t}(d_t (K)) -\sum_{i=1}^n G_i(\hat{c}_i^T \omega_{i,m}(K))-\frac{Lm\lambda^2}{K}\label{eq:thm-proof-k-4}
\end{align}
Now, we can use the definition of $\alpha$ to obtain:
\begin{align}
	H_i^* \big(\nabla H_i (\omega_{i,m}(K))\big)&\geq \alpha_{H_i}H_i (\omega_{i,m}(K))~\forall i\in[n]\label{eq:thm-proof-k-5}
\end{align}
For all $i\in[n]$, using the definition of $G_i$ and defining $\gamma_i :=\ln(1+\frac{U_i(e-1)}{L_i})$, we have:
\begin{align}
	-G'_i (u)+G_i (u)&=-G'_i (u)+\gamma_i G_i(u)-(\gamma_i -1)G_i(u)\nonumber\\
	&=\frac{L_i \gamma_i}{e-1}u-(\gamma_i -1)G_i(u)\label{eq:thm-proof-k-6}
\end{align} 
Combining $\ref{ell}$, $\ref{eq:thm-proof-k-4}$, $\ref{eq:thm-proof-k-5}$ and $\ref{eq:thm-proof-k-6}$ along with $P_{\rm gseq}\geq 0$, we obtain:
\begin{align*}
	{\rm D}^* -\frac{Lm\lambda^2}{K}&\leq \sum_{t=1}^m \sigma_{F_t} \big(d_t (K)\big)-\frac{Lm\lambda^2}{K}\\
	&-\sum_{i=1}^n \big(H_i^*(\nabla H_i (\omega_{i,m}(K)))+G'_i (\hat{c}_i^T\omega_{i,m}(K)))\\
	&\leq {\rm ALG}-\sum_{i=1}^n \alpha_{H_i}H_i (\omega_{i,m}(K))\\
	&+\sum_{i=1}^n\big(\frac{L_i \gamma_i}{e-1}\hat{c}_i^T\omega_{i,m}(K)-(\gamma_i -1)G_i (\hat{c}_i^T\omega_{i,m}(K))\big)\\
	&\leq {\rm ALG}-\sum_{i=1}^n \alpha_{H_i}H_i(\omega_{i,m}(K))+\sum_{i=1}^n(\frac{\gamma_i}{e-1}+\gamma_i -1)H_i (\omega_{i,m}(K))\\
	&\leq \max_{i\in[n]}\big\{1-\alpha_{H_i}+\frac{\gamma_i}{e-1}+\gamma_i -1\big\}{\rm ALG}\\
	&= \max_{i\in[n]}\big\{-\alpha_{H_i}+\gamma_i\frac{e}{e-1}\big\}{\rm ALG}\\
\end{align*}
Therefore, if $K\to \infty$, the competitive ratio would be derived as $\frac{{\rm ALG}}{{\rm D}^*}\geq \frac{1}{\max_{i\in[n]}\big\{-\alpha_{H_i}+\gamma_i\frac{e}{e-1}\big\}}$.\\
\section{Proof of Theorem 4.2}
Considering that $G'_1 (u)=\ln(\frac{U_1 e}{L_1})G_1 (u)~;u\geq \frac{1}{\ln\big(\frac{U_1 e}{L_1}\big)}$, combining $\ref{eq:thm-proof-k-4}$ and $\ref{eq:thm-proof-k-5}$ for $n=1$ along with $P_{\rm gseq}\geq 0$, we obtain:
\begin{align*}
	D^* -\frac{Lm\lambda^2}{K}&\leq \sum_{t=1}^m \sigma_{F_t} \big(d_t(K)\big)-\frac{Lm\lambda^2}{K}\\
	&-H_1^*(\nabla H_1 (\omega_{1,m}(K)))-G'_1 (\hat{c}_1^T\omega_{1,m}(K))\\
	&\leq {\rm ALG}-\alpha_{H_1}H_1 (\omega_{1,m}(K))-\ln(\frac{U_1}{L_1})G_1(\hat{c}_1^T\omega_{1,m}(K))\\
	&\leq {\rm ALG}-\alpha_{H_1}{\rm ALG}+\ln(\frac{U_1}{L_1}){\rm ALG}\\
	&= \big(1-\alpha_{H_1}+\ln(\frac{U_1}{L_1})\big){\rm ALG}\\
\end{align*}
Therefore, if $K\to \infty$, the competitive ratio would be derived as $\frac{1}{1-\alpha_{H_1}+\ln(\frac{U_1}{L_1})}$.

\begin{thebibliography}{10}
	
	\bibitem{obm-Karp}
	Richard~M Karp, Umesh~V Vazirani, and Vijay~V Vazirani.
	\newblock An optimal algorithm for on-line bipartite matching.
	\newblock In {\em Proceedings of the twenty-second annual ACM symposium on
		Theory of computing}, pages 352--358. ACM, 1990.
	
	\bibitem{adwords-Mehta}
	Aranyak Mehta, Amin Saberi, Umesh Vazirani, and Vijay Vazirani.
	\newblock Adwords and generalized online matching.
	\newblock {\em Journal of the ACM (JACM)}, 54(5):22, 2007.
	
	\bibitem{adwords-Buchbinder}
	Niv Buchbinder, Kamal Jain, and Joseph~Seffi Naor.
	\newblock Online primal-dual algorithms for maximizing ad-auctions revenue.
	\newblock In {\em European Symposium on Algorithms}, pages 253--264. Springer,
	2007.
	
	\bibitem{swm-Lehmann}
	Benny Lehmann, Daniel Lehmann, and Noam Nisan.
	\newblock Combinatorial auctions with decreasing marginal utilities.
	\newblock {\em Games and Economic Behavior}, 55(2):270--296, 2006.
	
	\bibitem{onlinecp-Buch}
	Niv Buchbinder and Joseph Naor.
	\newblock Online primal-dual algorithms for covering and packing.
	\newblock {\em Mathematics of Operations Research}, 34(2):270--286, 2009.
	
	\bibitem{online-Azaretal}
	Yossi Azar, Niv Buchbinder, TH~Hubert Chan, Shahar Chen, Ilan~Reuven Cohen,
	Anupam Gupta, Zhiyi Huang, Ning Kang, Viswanath Nagarajan, Joseph Naor,
	et~al.
	\newblock Online algorithms for covering and packing problems with convex
	objectives.
	\newblock In {\em Foundations of Computer Science (FOCS), 2016 IEEE 57th Annual
		Symposium on}, pages 148--157. IEEE, 2016.
	
	\bibitem{online-Eghbali}
	Reza Eghbali and Maryam Fazel.
	\newblock Worst case competitive analysis of online algorithms for conic
	optimization.
	\newblock {\em arXiv preprint arXiv:1611.00507}, 2016.
	
	\bibitem{primaldual-Buchbinder}
	Niv Buchbinder, Joseph~Seffi Naor, et~al.
	\newblock The design of competitive online algorithms via a primal--dual
	approach.
	\newblock {\em Foundations and Trends{\textregistered} in Theoretical Computer
		Science}, 3(2--3):93--263, 2009.
	
	\bibitem{auction-Kleinberg}
	Robert Kleinberg.
	\newblock A multiple-choice secretary algorithm with applications to online
	auctions.
	\newblock In {\em Proceedings of the sixteenth annual ACM-SIAM symposium on
		Discrete algorithms}, pages 630--631. Society for Industrial and Applied
	Mathematics, 2005.
	
	\bibitem{knapsack-Maehara}
	Takanori Maehara, Atsuhiro Narita, Jun Baba, and Takayuki Kawabata.
	\newblock Optimal bidding strategy for brand advertising.
	\newblock In {\em IJCAI}, pages 424--432, 2018.
	
	\bibitem{swm-Vondrak}
	Jan Vondr{\'a}k.
	\newblock Optimal approximation for the submodular welfare problem in the value
	oracle model.
	\newblock In {\em Proceedings of the fortieth annual ACM symposium on Theory of
		computing}, pages 67--74. ACM, 2008.
	
	\bibitem{dr-Bian}
	Andrew~An Bian, Baharan Mirzasoleiman, Joachim~M Buhmann, and Andreas Krause.
	\newblock Guaranteed non-convex optimization: Submodular maximization over
	continuous domains.
	\newblock {\em arXiv preprint arXiv:1606.05615}, 2016.
	
	\bibitem{onlinematrix-Eghbali}
	Reza Eghbali, James Saunderson, and Maryam Fazel.
	\newblock Competitive online algorithms for resource allocation over the
	positive semidefinite cone.
	\newblock {\em Mathematical Programming}, pages 1--26, 2018.
	
	\bibitem{multilinearmax-calinescu}
	Gruia Calinescu, Chandra Chekuri, Martin P{\'a}l, and Jan Vondr{\'a}k.
	\newblock Maximizing a submodular set function subject to a matroid constraint.
	\newblock In {\em International Conference on Integer Programming and
		Combinatorial Optimization}, pages 182--196. Springer, 2007.
	
	\bibitem{sub-Iyer}
	Rishabh Iyer, Stefanie Jegelka, and Jeff Bilmes.
	\newblock Monotone closure of relaxed constraints in submodular optimization:
	Connections between minimization and maximization: Extended version.
	\newblock 2014.
	
	\bibitem{drapp-Bian}
	An~Bian, Joachim~M Buhmann, and Andreas Krause.
	\newblock Optimal dr-submodular maximization and applications to provable mean
	field inference.
	\newblock {\em arXiv preprint arXiv:1805.07482}, 2018.
	
	\bibitem{dralg-Bian}
	An~Bian, Kfir Levy, Andreas Krause, and Joachim~M Buhmann.
	\newblock Continuous dr-submodular maximization: Structure and algorithms.
	\newblock In {\em Advances in Neural Information Processing Systems}, pages
	486--496, 2017.
	
	\bibitem{auction-Huang}
	Zhiyi Huang and Anthony Kim.
	\newblock Welfare maximization with production costs: A primal dual approach.
	\newblock {\em Games and Economic Behavior}, 2018.
	
	\bibitem{knapsack-Zhou}
	Yunhong Zhou, Deeparnab Chakrabarty, and Rajan Lukose.
	\newblock Budget constrained bidding in keyword auctions and online knapsack
	problems.
	\newblock In {\em International Workshop on Internet and Network Economics},
	pages 566--576. Springer, 2008.
	
	\bibitem{lacoste2016convergence}
	Simon Lacoste-Julien.
	\newblock Convergence rate of frank-wolfe for non-convex objectives.
	\newblock {\em arXiv preprint arXiv:1607.00345}, 2016.
	
	\bibitem{curvature-Conforti}
	Michele Conforti and G{\'e}rard Cornu{\'e}jols.
	\newblock Submodular set functions, matroids and the greedy algorithm: tight
	worst-case bounds and some generalizations of the rado-edmonds theorem.
	\newblock {\em Discrete applied mathematics}, 7(3):251--274, 1984.
	
	\bibitem{drmax-Calinescu}
	Gruia Calinescu, Chandra Chekuri, Martin P{\'a}l, and Jan Vondr{\'a}k.
	\newblock Maximizing a monotone submodular function subject to a matroid
	constraint.
	\newblock {\em SIAM Journal on Computing}, 40(6):1740--1766, 2011.
	
	\bibitem{knapsack-Kulik}
	Ariel Kulik, Hadas Shachnai, and Tami Tamir.
	\newblock Maximizing submodular set functions subject to multiple linear
	constraints.
	\newblock In {\em Proceedings of the twentieth annual ACM-SIAM symposium on
		Discrete algorithms}, pages 545--554. Society for Industrial and Applied
	Mathematics, 2009.
	
	\bibitem{gsubmax-chekuri}
	Chandra Chekuri, Jan Vondr{\'a}k, and Rico Zenklusen.
	\newblock Submodular function maximization via the multilinear relaxation and
	contention resolution schemes.
	\newblock {\em SIAM Journal on Computing}, 43(6):1831--1879, 2014.
	
	\bibitem{mwu-Chekuri}
	Chandra Chekuri, TS~Jayram, and Jan Vondr{\'a}k.
	\newblock On multiplicative weight updates for concave and submodular function
	maximization.
	\newblock In {\em Proceedings of the 2015 Conference on Innovations in
		Theoretical Computer Science}, pages 201--210. ACM, 2015.
	
	\bibitem{gsubmax-Feldman}
	Moran Feldman, Joseph Naor, and Roy Schwartz.
	\newblock A unified continuous greedy algorithm for submodular maximization.
	\newblock In {\em Foundations of Computer Science (FOCS), 2011 IEEE 52nd Annual
		Symposium on}, pages 570--579. IEEE, 2011.
	
	\bibitem{gsubmax-Ene}
	Alina Ene and Huy~L Nguyen.
	\newblock Constrained submodular maximization: Beyond 1/e.
	\newblock In {\em Foundations of Computer Science (FOCS), 2016 IEEE 57th Annual
		Symposium on}, pages 248--257. IEEE, 2016.
	
	\bibitem{gsubmax-Buch}
	Niv Buchbinder and Moran Feldman.
	\newblock Constrained submodular maximization via a non-symmetric technique.
	\newblock {\em arXiv preprint arXiv:1611.03253}, 2016.
	
	\bibitem{survey-Krause}
	Andreas Krause and Daniel Golovin.
	\newblock Submodular function maximization., 2014.
	
	\bibitem{survey-Buchbinder}
	Niv Buchbinder and Moran Feldman.
	\newblock Submodular functions maximization problems.
	\newblock 2017.
	
	\bibitem{knapsack-Marchetti}
	Alberto Marchetti-Spaccamela and Carlo Vercellis.
	\newblock Stochastic on-line knapsack problems.
	\newblock {\em Mathematical Programming}, 68(1-3):73--104, 1995.
	
	\bibitem{secretary-Bateni}
	MohammadHossein Bateni, Mohammadtaghi Hajiaghayi, and Morteza Zadimoghaddam.
	\newblock Submodular secretary problem and extensions.
	\newblock {\em ACM Transactions on Algorithms (TALG)}, 9(4):32, 2013.
	
	\bibitem{secretary-Gupta}
	Anupam Gupta, Aaron Roth, Grant Schoenebeck, and Kunal Talwar.
	\newblock Constrained non-monotone submodular maximization: Offline and
	secretary algorithms.
	\newblock In {\em International Workshop on Internet and Network Economics},
	pages 246--257. Springer, 2010.
	
	\bibitem{secretary-Kesselheim}
	Thomas Kesselheim and Andreas T{\"o}nnis.
	\newblock Submodular secretary problems: Cardinality, matching, and linear
	constraints.
	\newblock {\em arXiv preprint arXiv:1607.08805}, 2016.
	
\end{thebibliography}
\bibliographystyle{unsrtnat}

\end{document}